\def\changesHilighted{false} %コメントオフ：投稿バージョン
\documentclass[final,3p,times]{elsarticle}
\usepackage{amssymb}
\usepackage[dvipsnames]{xcolor}
\usepackage{mathtools}
\usepackage{svg}
\usepackage{amsmath}
\DeclarePairedDelimiter{\abs}{\lvert}{\rvert}
\newcommand{\norm}[1]{\lVert #1 \rVert}

\definecolor{dogwoodrose}{rgb}{0.84, 0.09, 0.41}
\definecolor{darkpastelgreen}{rgb}{0.01, 0.75, 0.24}

\newcommand{\add}[1]{\textcolor{\addcolor}{#1}}
\newcommand{\clyde}[1]{\textbf{\textcolor{RawSienna}{\ [Clyde: #1]}}}
\newcommand{\masaki}[1]{\textbf{\textcolor{dogwoodrose}{\textbf{\ [Masaki: #1]}}}}

\usepackage[normalem]{ulem}

\newcommand{\del}[1]{\textcolor{red}{\sout{#1}}}

\newcommand{\dell}[1]{\textcolor{red}{\sout{#1}}}

\newcommand{\delwhole}[1]{{#1}}

\ifnum\pdfstrcmp{\changesHilighted}{false}=0
\renewcommand{\add}[1]{#1}

\renewcommand{\del}[1]{}
\renewcommand{\dell}[1]{}
\renewcommand{\masaki}[1]{}
\renewcommand{\clyde}[1]{}
\renewcommand{\delwhole}[1]{}
\fi

  \newcommand{\gwidth}{9.1cm}

\newdefinition{definition}{Definition}[section]
\newtheorem{assumption}[definition]{Assumption}
\newtheorem{lemma}[definition]{Lemma}

\newtheorem{theorem}[definition]{Theorem}
\newdefinition{remark}[definition]{Remark}
\newdefinition{example}[definition]{Example}

\newcommand{\Norm}[1]{\left\lVert #1 \right\rVert}
\newcommand{\Abs}[1]{\left| #1 \right|}

\DeclareMathOperator{\Gl}{Gl} %General Linear group
\DeclareMathOperator{\Sl}{Sl} %Special Linear group
\DeclareMathOperator{\tr}{tr} %TRace of a matrix
\DeclareMathOperator{\Sp}{Sp} %DIAGonal matrix

\newproof{proof}{Proof}
\newproof{pf}{Proof}

\newcommand{\mywidth}{8.7cm}
\newcommand{\myvecwidth}{8.7cm}

\newcommand{\nnorm}[1]{{\left\vert\kern-0.25ex\left\vert\kern-0.25ex\left\vert #1 \right\vert\kern-0.25ex\right\vert\kern-0.25ex\right\vert}}

\usepackage{algorithm}
\usepackage{algorithmic}

\usepackage{mathtools}

\begin{document}

\begin{frontmatter}



\title{Mean Escape Time of Switched Riccati Differential Equations}


\author[OU,TTU]{Masaki Ogura\corref{cor1}}
\corref{cor1}
\ead{m-ogura@ist.osaka-u.ac.jp}
\cortext[cor1]{Corresponding author}
\author[TTU]{Clyde Martin}

\affiliation[OU]{organization={Graduate School of Information Science and Technology, Osaka University},%Department and Organization
            addressline={Yamadaoka 1-5}, 
            city={Suita},
            postcode={565-0871}, 
            state={Osaka},
            country={Japan}}

\affiliation[TTU]{organization={Department of Mathematics and Statistics, Texas Tech University},%Department and Organization
            addressline={1108 Memorial Circle}, 
            city={Lubbock},
            postcode={79409}, 
            state={TX},
            country={USA}}

\begin{abstract}
Riccati differential equations is the class of first-order and quadratic ordinary differential equations and has various applications in the systems and control theory. In this paper, we analyze a switched Riccati differential equation that is driven by a Poisson-like stochastic signal. We specifically focus on the computation of the mean escape time of the switched Riccati differential equation. The contribution of this paper is twofold. We first show that, under the assumption that the subsystems described as a deterministic Riccati differential equation escape in finite time regardless of its initial state, the mean escape time of the switched Riccati differential equation admits a power series expression. In order to further expand the applicability of this result, we then present an approximative formula for computing the escape time of deterministic Riccati differential equations. We present numerical simulations to illustrate the obtained results. 
\end{abstract}


% \begin{highlights}
% \end{highlights}


\end{frontmatter}


\section{Introduction}

\clyde{dear Masaki, I read the aper this morning and I like it. In the introduction o the Riccati equation and the Grassmannian you might add that the finite escape time of the differential equation amounts to leaving the canonical chart on the Grassmannian. You might also point out that by a change of basis every nonlinear riccati equation is equivalent to a linear ODE or to a forced linear ODE. You might wnat to mention in the intro that there is a huge theory of scalar time varying Riccati equations as developed in Watson's "Treatise on the Bessel Functions." I don't think that material has been approached in a modern setting. It is very complicated being equivalent to the theory of special functions..}\masaki{Thanks for your comments. I have updated the first paragraph of the introduction accordingly.}
\add{Riccati differential equations (RDEs)~\cite{Shayman1986} is the class of first-order and quadratic ordinary differential equations and has various applications in}
\del{The applications of Riccati differential equations (RDEs) \cite{Shayman1986} include }dynamic games~\cite{Basar1989},
$H^\infty$ optimal control problems~\cite{Doyle1989}, and singular
perturbation of boundary value problems~\cite{Chang1972}. 
For a recent survey of RDEs, see
\cite{Freiling2002} and the references therein. \add{For the specific case of scalar RDEs, we refer the readers to the classic work by Watson~\cite{Watson1995}.}
\add{Although it is well-known that a Riccati differential equation is locally equivalent to a linear ordinary differential equation by a change of basis~\cite{Shayman1986},} \del{
Being a nonlinear differential equation, }\add{the non-linearlity of} a RDE \add{allows its}
\del{can have the }solution \add{to diverge}\del{that diverges} in a finite time, which is called
the finite-time escape phenomena. Various \add{characterizations and} conditions\del{ under which the
solution of a RDE escapes in a} \add{for the} finite\add{-}time \add{escape phenomena} can
be found in \add{the literature (see, e.g., }\cite{Martin1981,Sasagawa1982,Crouch1987}\add{)}. \add{Among the various results, the most fundamental fact is that the finite-time escape can be characterized by the leave of the induced flow on a manifold called the Grassmannian~\cite{Shayman1986,Doolin1990} from its canonical chart.} \add{As for the qualitative analysis of the finite-time escape phenomena,} 
Getz and Jacobson~\cite{Getz1977} derive a sufficient condition for the solution
to escape in a finite time, and also give an upper bound of the escape
time. On the other hand, Freiling et al.~\cite{Freiling2000} give a condition under which
a finite-time escape phenomena does not occur. 

The primary objective of this paper is investigating the escape time of \emph{switched} RDEs, in which its subsystems are described by RDEs, and its switching dynamics is driven by a Poisson-like stochastic signal. It can be easily confirmed that a switched linear system with Poisson jumps is a Markov jump linear system, which is widely studied in the literature (see, e.g.,~\cite{Zhang2008,Feng1992,Zhang2010,Shi2015}). Although we can find extensive amount of works for studying the Markov jump nonlinear systems~\cite{Wu2017a,Jin2022,Zhang2019b,Shen2016b}, most of them implicitly assume the non-existence of the finite-time escape in each of the subsystems. Therefore, these results available in the literature for Markov jump nonlinear systems are not directly applicable to switched RDEs studied in this paper. 

\masaki{I have chosen to omit the part on switched ERDE (previously Section 5) because of the high iThenticate duplicate score of 50\%. This value is high enough for an Editor-reject. By removing the part coming from CDC2012, I could reduce the score by 10\%, which is a big number. I have also rephrased various parts of the paper to eventually get an acceptable score of 34\%. I'm going to carefully go over the draft once again to make sure that we did not lose consistency due to this big change.}

To fill in the aforementioned gap, in this paper, we first study the computation of the mean escape time of a switched RDE subject to Poisson switching. 
Specifically, under the assumption that its subsystems described by deterministic RDEs escape in finite time regardless of their initial states, we show that the mean escape time of the switched RDE admits a power series expression. In order to further expand the applicability of this result, we then present an approximative formula for the computation of the escape time of deterministic RDEs.

This paper is organized as follows. After preparing necessary mathematical
notation and conventions, in Section~\ref{section:Grass} we briefly review the RDEs and Grassmannians and, then, introduce the switched RDEs subject to Poisson switching. Then, in Section~\ref{section:ExpectedEscapeTime}, we study the
expected value of the escape time of switched RDEs. We then study the approximate computation of the escape time of both deterministic and switched RDEs in Section~\ref{subsec:approx}. 
We finally conclude the paper in Section~\ref{sec:conc}.

%%%%%%%%%%%%%%%%%%%%%%%%%%%%%%%%%%%%%%%%%%%%%%%%%%%%%%%%%%%%%%%%%%%%%%%%%%%%%%%%
\subsection{Notation}

We let the field of real numbers be denoted by $\mathbb{R}$. For a positive  number $x$, we define
$\log^+(x) = \max(\log x, 0)$. For a subset
$S$ on $\mathbb{R}$, we let $\chi_{S}$ denote the characteristic function of
$S$. 
product in~$\mathbb{R}^d$, which yields 
The Euclidean norm of a real vector~$v$ is denoted by $\norm{v}$. We also define the norm
$\norm{v}_\infty = \max_{i=1}^d \abs{v_i}$. 
We denote by $I_n$ the identity matrix of the dimension $n$. The
subscript $n$ will be omitted when it is clear from the context. 
The maximum singular value of a matrix~$M \in \mathbb{R}^{m\times d}$ is denoted
by~$\norm{M}$. We let $\Sp M$ denote the column space of~$M$. When $M$ is square, we let $\tr M$ denote the trace of~$M$. For a subspace~$V$
in~$\mathbb{R}^d$, we define the subspace~$M(V)$ of~$\mathbb{R}^m$ by
% \begin{equation*}
  $M(V) = \{Mv: v\in V\}$.
% \end{equation*}
Let $\Gl(d, \mathbb{R})$ denote the multiplicative group of invertible
$d \times d$ real matrices. The subgroup of  $\Gl(d, \mathbb{R})$ consisting of the matrices having
determinant~$1$ is denoted by~$\Sl(d, \mathbb{R})$.

For a measure space~$X$, we let $L^\infty(X)^n$ denote the set of 
$\mathbb{R}^n$\nobreakdash-valued Lebesgue measurable functions~$f$ on~$X$ satisfying
% \begin{equation*}
  $\sup_{x \in X} \norm{f(x)}_\infty < \infty$.
% \end{equation*}
The set $L^\infty(X)^n$ becomes a Banach space when equipped with the
norm~$\norm{f} = \sup_{x \in X} \norm{f(x)}_\infty$. 

Let $X$ be metric space with a distance $\rho$. A subset $N\subset X$ is called an
$\epsilon$-net of~$X$ if for every $x\in X$ there exists $y\in N$ such
that $\rho(x, y) < \epsilon$. $X$ is called totally bounded if $X$
admits a finite $\epsilon$-net for every $\epsilon>0$. 

For a Banach
space~$X$, let $\mathcal{L}(X)$ denote the space of continuous linear operators
on~$X$. The identity operator in $\mathcal{L}(X)$ is denoted by $I$. The set~$\mathcal{L}(X)$ becomes a Banach space when equipped with the norm
% \begin{equation*}
  $\norm{A} = \sup_{x \neq 0} {\norm{Ax}}/{\norm{x}}$.
% \end{equation*}
The next lemma about the invertibility of the operators in $\mathcal{L}(X)$ is well known.

\begin{lemma} \label{lemma:SmallGainBanach}
  Let $X$ be a Banach space. Take an arbitrary $A \in \mathcal{L}(X)$. Assume $\norm{A} < 1$. 
  Then, the operator $I-A$ is invertible in~$\mathcal{L}(X)$ and, furthermore, the inverse equals $\sum_{k=0}^\infty A^k$.
\end{lemma}

Let us also state a matrix version of Lemma~\ref{lemma:SmallGainBanach}. 

\begin{lemma} \label{lemma:SmallGain}
Let $M$ be a square matrix. 
\begin{enumerate}
\item  If $\norm{M - I} < 1$,  then $M$ is invertible. 

\item  If $\norm M < 1$, then the matrix~$I-M$ is invertible and, moreover, 
$\norm{(I-M)^{-1}} \leq 1/(1-\norm M)$. 
\end{enumerate}
\end{lemma}

%%%%%%%%%%%%%%%%%%%%%%%%%%%%%%%%%%%%%%%%%%%%%%%%%%%%%%%%%%%%%%%%%%%%%%%%%%%%%%%%
\section{Switched Riccati Differential Equations}
\label{section:Grass}

In this section, we introduce the RDE with Poisson jumps. In Subsection~\ref{subsec:RDE}, we briefly review RDEs with an emphasis on their relationship to Grassmannian manifolds. Then, in Subsection~\ref{subsc:RDEMJ}, we formulate RDEs with Poisson jumps, which is the main subject of this paper. 

\subsection{RDEs and Grassmannians}\label{subsec:RDE}

Let us give a brief review on RDEs and their relationship to Grassmannians~\cite{Shayman1986}. Take an arbitrary matrix~$A\in\mathbb{R}^{d\times d}$. Let $k < d$ be a positive integer. Partition the matrix~$A$ as 
\begin{equation*}
  A =
  \begin{bmatrix}
    A_{11}&A_{12}\\
    A_{21}&A_{22}
  \end{bmatrix}, 
\end{equation*}
for $A_{11} \in \mathbb{R}^{k\times k}$, $A_{12} \in \mathbb{R}^{k \times (d-k)}$, $A_{21}\in \mathbb{R}^{(d-k) \times k}$, and~$A_{22}\in \mathbb{R}^{(d-k) \times (d-k)}$. 
Then, we can define the Riccati differential equation (RDE) associated with the matrix $A$ by 
\begin{equation} \label{eq:RDE}
  \frac{dY}{dt} = A_{21} + A_{22}Y - Y A_{11} - YA_{12}Y. 
\end{equation}
We remark that this RDE is not necessarily symmetric because the matrix $A$ is taken arbitrarily.
Now, we let the solution of this RDE with the initial
condition~$Y(0) = Y_0$ by $Y(\cdot;Y_0)$. Then, the first time $t > 0$ at which the solution~$Y(t;Y_0)$ does not exist is defined as
the {\it escape time}~\cite{Martin1981} of the RDE~\eqref{eq:RDE}.
 When such $t$ does not exist, i.e., if the solution of the RDE~\eqref{eq:RDE} exists for all $t\geq 0$, then we regard the escape time of the RDE as $+\infty$. 

RDEs are closely related to manifolds called
Grassmannians~\cite{Shayman1986,Martin1981}. Specifically, the Grassmannian~$G^k(\mathbb{R}^d)$ consists
of the set of all the $k$\nobreakdash-dimensional subspaces
in~$\mathbb{R}^d$. In the special case of $k=1$, the Grassmannian reduces to the real
projective space and is denoted by~$P(\mathbb{R}^d)$. 
Notice that the group~$\Gl(\mathbb{R}^d)$ acts on the manifold~$G^k(\mathbb{R}^d)$ because an invertible matrix
always maps a $k$\nobreakdash-dimensional subspace to another $k$\nobreakdash-dimensional
subspace.

There exists a fundamental connection between the RDEs and the Grassmannians. Let us define $\psi \colon \mathbb{R}^{(d-k) \times k} \to G^k(\mathbb{R}^d)$ by the equation
\begin{equation*}
  \psi(K) = \Sp\begin{bmatrix} I_k \\ K \end{bmatrix}. 
\end{equation*}
We also define $G_0^k(\mathbb{R}^d) \subset G^k(\mathbb{R}^d)$ as the set of all the 
$k$\nobreakdash-dimensional subspaces in~$\mathbb{R}^d$ complementary
to the subspace~$\Sp\begin{bmatrix}0\\I_{d-k}\end{bmatrix}$. Then, one can see that the mapping~$\psi$
embeds the set~$\mathbb{R}^{(d-k)\times k}$ into the Grassmannian~$G^k(\mathbb{R}^d)$ as the open and
dense subset~$G_0^k(\mathbb{R}^d)$. 
Furthermore, one can show that the equation 
\begin{equation} \label{eq:RDE--ERDE}
  \psi(Y(t;Y_0)) = e^{At}(\psi(Y_0))
\end{equation}
holds true whenever the solution~$Y(t;Y_0)$ of the RDE~\eqref{eq:RDE} exists. This equation implies that RDE~\eqref{eq:RDE} can be regarded as the local
expression of the differential equation on~$G^k(\mathbb{R}^d)$ corresponding to
the flow
\begin{equation} \label{eq:ERDE}
  F(t;F_0) = e^{At}(F_0)
\end{equation}
with respect to the chart~$(G_0^k(\mathbb{R}^d), \psi^{-1})$. Therefore, by the
{\it extended RDE} (ERDE), we mean the differential
equation on~$G^k(\mathbb{R}^d)$ having the flow~\eqref{eq:ERDE}.
For this reason, with the abuse of notation, we write the solution of RDE~\eqref{eq:RDE} as
\begin{equation}\label{eq:notation}
  e^{A t}.Y_0 = Y(t;Y_0).
\end{equation}
It should be noticed that, by the relationship~\eqref{eq:RDE--ERDE}, we can see that RDE~\eqref{eq:RDE} escapes
exactly when the flow $e^{At}(\psi(Y_0))$ leaves the subset~$G_0^k(\mathbb{R}^d)$.

As for the escape time of the RDE~\eqref{eq:RDE}, we can specifically
see that the solution of  \eqref{eq:RDE} is given by $Y(t;Y_0) = V(t)
U(t)^{-1}$ where $U(t) \in \mathbb{R}^{k\times k}$ and $V(t) \in
\mathbb{R}^{(d-k)\times k}$ are given by the partition 
\begin{equation}\label{eq:def:UV}
\begin{bmatrix}
U(t)\\V(t)
\end{bmatrix}
 = e^{At}\begin{bmatrix}
I _n\\ Y_0
\end{bmatrix} 
\end{equation}
provided that $U(t)^{-1}$ exists. Hence, the equation \eqref{eq:RDE} 
escapes precisely at the minimum $t\geq 0$ at which the matrix $U(t)$
becomes singular. 

Let us see an example to fix ideas. 

\begin{example} %\label{example:RDE--ERDE}
  Let $\omega$ be a positive number. Let us consider the RDE
  \begin{equation} \label{eq:ExampleRDE1}
    \frac{dy}{dt} = \omega (1+y^2). 
  \end{equation}
  We set the initial condition of this RDE as $y(0) = 0$. 
  It is trivial to see that the solution of this RDE is given by $y(t) = \tan(\omega t)$. This implies that the
    RDE~\eqref{eq:ExampleRDE1} escapes at the time $t = \pi/(2\omega)$.
  On the other hand, the ERDE induced by the RDE~\eqref{eq:ExampleRDE1} can be derived as follows. 
  First, to the domain~$\mathbb{R}$ of the RDE~\eqref{eq:ExampleRDE1}, we can 
  correspond the Grassmannian~$G^1(\mathbb{R}^2)$ (i.e., the projective
  space~$P(\mathbb{R}^2)$). Hence, the canonical chart~$\psi$ maps a real number 
  $t$ to the straight line having
  the slope~$t$ and passing through the origin. This particularly implies that the initial state $y(0) = 0$ of the RDE is mapped to the
  $x$\nobreakdash-axis. Furthermore, with this canonical chart, a flow on~$P(\mathbb{R}^2)$
  escapes exactly when it coincides with the $y$\nobreakdash-axis.
  Now, because the RDE~\eqref{eq:ExampleRDE1} is induced by the matrix
  \begin{equation*}
    A =
    \begin{bmatrix}
    0     &-\omega\\
    \omega&0
    \end{bmatrix},
  \end{equation*} 
  the ERDE is nothing but the flow described by 
  \begin{equation*}
    e^{At}.(\psi(0))
    =
    \left\{
    r\begin{bmatrix}
    \cos(\omega t) \\ \sin(\omega t)
    \end{bmatrix}
    :
    r\in \mathbb{R}
    \right\}, 
  \end{equation*}
  representing a straight line rotating with the angular speed~$\omega$ counterclockwise.
  Therefore, starting initially from the $x$\nobreakdash-axis,
  the flow escapes at time~$\pi/(2\omega)$, which coincides with our finding above. 
\end{example}

\subsection{RDEs subject to Poisson switching}\label{subsc:RDEMJ}

The objective of this paper is to study the escape time of switched RDEs subject to Poisson jumps, which we describe below. 
Let $A, B\in \mathbb{R}^{d\times d}$ and consider the RDE~\eqref{eq:RDE} and another RDE
\begin{equation}\label{eq:BRDE}
  \frac{dY}{dt} = B_{21} + B_{22}Y -  YB_{11} - YB_{12}Y. 
\end{equation}
Although we focus on the case of two subsystems, most of the results presented in this paper can be extended to the case of more than two subsystems. 

We define a $\{0, 1\}$\nobreakdash-valued random variable~$z$ by the
stochastic differential equation~\cite{Hanlon2011a}
\begin{equation} \label{eq:z:SwitchingSignal}
  dz = (1-2z)dN,\quad z(0) \in \{0, 1\}
\end{equation}
where $N$ is the Poisson process of rate~$\lambda>0$. The stochastic process $z$ is a continuous-time Markov process with switching rate~$\lambda$. 

Now, we consider the following RDE with Poisson jump: 
\begin{equation} \label{eq:SwitchedRDE}
  \begin{multlined}
    dY
    = \left[z (A_{21} + A_{22}Y - Y A_{11} - YA_{12}Y) +%\right. 
    % \\
    %\left.
    (1-z)(B_{21} + B_{22}Y - Y B_{11} - YB_{12}Y)\right]dt. 
  \end{multlined}
\end{equation}
Our major objective in this paper is in the numerical calculation and the theoretical characterization of the average escape time of the switched RDE~\eqref{eq:SwitchedRDE}. Specifically, let $Y_0 \in \mathbb{R}^{(d-k)\times k}$ be arbitrary, and let $T_A(Y_0)$ ($T_B(Y_0)$) denote the \emph{mean escape time} of the switched RDE~\eqref{eq:SwitchedRDE} when $Y(0) = Y_0$ and $z(0) = 1$ ($z(0) = 0$, respectively). In Section~\ref{section:ExpectedEscapeTime}, we show that the mean escape times admit a representation as a power series, under the assumption that the escape time of the deterministic RDEs~\eqref{eq:RDE} and \eqref{eq:BRDE} are analytically available. Then, in Section~\ref{subsec:approx}, we show that, even when the escape time of the deterministic RDEs are not available, we can still approximately compute the mean escape times. 

%%%%%%%%%%%%%%%%%%%%%%%%%%%
\section{Analytical Charecterization of Mean Escape Time}
\label{section:ExpectedEscapeTime}

Let $t_A(Y_0)$ and $t_B(Y_0)$ denote the escape time of the deterministic RDEs~\eqref{eq:RDE} and \eqref{eq:BRDE}. Throughout this section, we place the following assumption. 

\begin{assumption}\label{asm:boudnedneddsss}
Assume that there exists a constant~$t_0 > 0$ such that
  \begin{equation} \label{eq:EscTimeBounded}
    t_i(Y_0) \leq t_0
  \end{equation}  
  for all $Y_0 \in \mathbb{R}^{(d-k)\times k}$ and $i \in \{A, B\}$. 
\end{assumption}

The main objective of this section is to prove the following theorem characterizing the mean escape time of the switched RDE~\eqref{eq:SwitchedRDE}. 

\begin{theorem}\label{thm:analytic}
Define the probability distribution function
% \begin{equation*}
  $F(t) = \int_{0}^t f(\tau)\,d\tau$
($t\geq 0$)
% \end{equation*}
associated with the probability density function
% \begin{equation*}
  $f(t) = \lambda e^{-\lambda t}$ ($t\geq 0$). Define the real valued functions $g_1$ and $g_2$ defined on $\mathbb{R}^{(d-k) \times k}$  by 
\begin{equation*}% \label{eq:def g_1}
\begin{aligned}
  g_1(Y) &= \int_0^{t_A(Y)} \tau f(\tau) \,d\tau
  +
  t_A(Y) \left(1-F(t_A(Y))\right), 
  \\
    g_2(Y) &= \int_0^{t_B(Y)}\tau f(\tau)\,d\tau
    +
    t_B(Y) \left(1-F(t_B(Y))\right). 
\end{aligned}
\end{equation*}
Also, define the operators $M_1$ and $M_2$ acting on $L^\infty(\mathbb{R}^{(d-k) \times k})$ by 
\begin{equation*}% \label{eq:def M_1}
\begin{aligned}
  (M_1T)(Y) &= \int_0^{t_A(Y)}f(\tau) T(e^{A \tau}.Y)\,d\tau, 
  \\
  (M_2T)(Y) &= \int_0^{t_B(Y)}f(\tau) T(e^{B \tau}.Y)\,d\tau
  \end{aligned}
\end{equation*}
for all $Y \in \mathbb{R}^{(d-k) \times k}$.
  Then, we have 
  \begin{equation} \label{eq:PowerSeriesSolution}
    \begin{bmatrix}
      T_A\\T_B
    \end{bmatrix}
    =
    \sum^{\infty}_{k=0}M^k\begin{bmatrix}
      g_1\\g_2
    \end{bmatrix}, 
  \end{equation}
  where the operator~$M$ is given by 
  \begin{equation*}
  M =
  \begin{bmatrix}
    0   & M_1\\
    M_2 & 0
  \end{bmatrix}.
  \end{equation*}
%   for every $i=A, B$ 
\end{theorem}

\begin{proof}
Let us temporarily assume $z(0) = 1$, i.e., let us suppose that the switched RDE~\eqref{eq:SwitchedRDE} initially starts from the first subsystem given in~\eqref{eq:RDE}. If no switch occurs after the initial time~$t=0$, then the solution of the RDE would escape at~$t = t_A(Y_0)$ by the definition of the mapping~$t_A$. On the other hand, suppose that the first switching occurs at a time~$t = \tau$. This $\tau$ must satisfy $\tau < t_A(Y_0)$. At this time instant, the solution of the switched RDE is equal to $e^{A \tau}.Y_0$ under the notation~\eqref{eq:notation}. Hence, by definition, the switched RDE will escape after $T_B(e^{A\tau}.Y_0)$ time units in average. Summarizing this argument, we obtain the following integral equation: 
\begin{equation*}
  \begin{multlined}
  T_A(Y_0)
  =
  \int^{t_A(Y_0)}_{0} f(\tau) [\tau + T_B(e^{A \tau}.Y_0)]\, d\tau
  +
  \left(1-F(t_A(Y_0))\right)t_A(Y_0), 
  \end{multlined}
\end{equation*}
which leads us to the functional equation
  $T_A = g_1 + M_1T_B$. 
In the same way, we can derive the functional equation
  $T_B = g_2 + M_2T_A$. 
Therefore, the mean escape times $T_A$ and~$T_B$ satisfy
\begin{equation*} 
  (I - M)\begin{bmatrix}
    T_A\\T_B
  \end{bmatrix}
  =
  \begin{bmatrix}
    g_1\\g_2
  \end{bmatrix}. 
\end{equation*}
Hence, by Lemma~\ref{lemma:SmallGainBanach}, to complete the proof of the theorem, 
 it is enough to show that the following claims hold true: 
\renewcommand{\labelenumi}{\alph{enumi})}
\begin{enumerate}
\item
$M$ is a continuous linear operator on the space~$L^\infty(\mathbb{R}^{(d-k) \times k})^2$;

\item
There holds that $\norm{M} < 1$.
\end{enumerate}

First, we can trivially confirm that $M$ is linear. To show its continuity, let us take arbitrary $T, T'\in L^\infty(\mathbb{R}^{(d-k) \times k})$. Then, we can show 
\begin{align}
\Norm{M\begin{bmatrix}T\\T'\end{bmatrix}}
&=
\sup_{W \in \mathbb{R}^{(d-k) \times k}}
\Norm{M\begin{bmatrix}T\\T'\end{bmatrix}(Y)}_\infty
\notag
\\
&=
\sup_{W \in \mathbb{R}^{(d-k) \times k}}
\Norm{\begin{bmatrix}
(M_1 T')(Y) \\
(M_2 T)(Y)
\end{bmatrix}}_\infty. 
\label{eq:norm<1 pre}
\end{align}
Hence, for every $Y\in \mathbb{R}^{(d-k) \times k}$, we can show 
\begin{align}
\Abs{(M_1T')(Y)}
&=
\Abs{\int_0^{t_A(Y)}f(\tau) T'(e^{A \tau}.Y)\,d\tau}
\notag
\\
&\leq
\int_0^{t_A(Y)}f(\tau) \abs{T'(e^{A \tau}.Y)}\,d\tau
\notag
\\
&\leq
\sup_{W \in \mathbb{R}^{(d-k) \times k}}{\abs{T'(Y)}}
\int_0^{t_0}f(\tau)\,d\tau
\notag
\\
&=
\norm{T'} F(t_0).
\label{eq:norm<1 1}
\end{align}
In the same way we can show 
\begin{equation}\label{eq:norm<1 2}
\Abs{(M_2T)(Y)} \leq \norm{T}F(t_0).
\end{equation}
The equation~\eqref{eq:norm<1 pre} together with \eqref{eq:norm<1 1} and
\eqref{eq:norm<1 2} yield
\begin{align}
\Norm{M\begin{bmatrix}T\\T'\end{bmatrix}}
&\leq
F(t_0)\Norm{\begin{bmatrix}\norm{T'}\\\norm{T}\end{bmatrix}}_\infty
\notag
\\
&=
F(t_0)\Norm{\begin{bmatrix}T\\T'\end{bmatrix}}. 
\label{eq:norm<1}
\end{align}
Therefore, we conclude that $M$ is a continuous linear operator. Furthermore, the second claim~b) follows immediately
from inequality~\eqref{eq:norm<1} because we know $F(t_0) < 1$. This completes the proof. 
\end{proof}

\begin{figure}[tbp]
\begin{center}
\includegraphics[width=\gwidth]{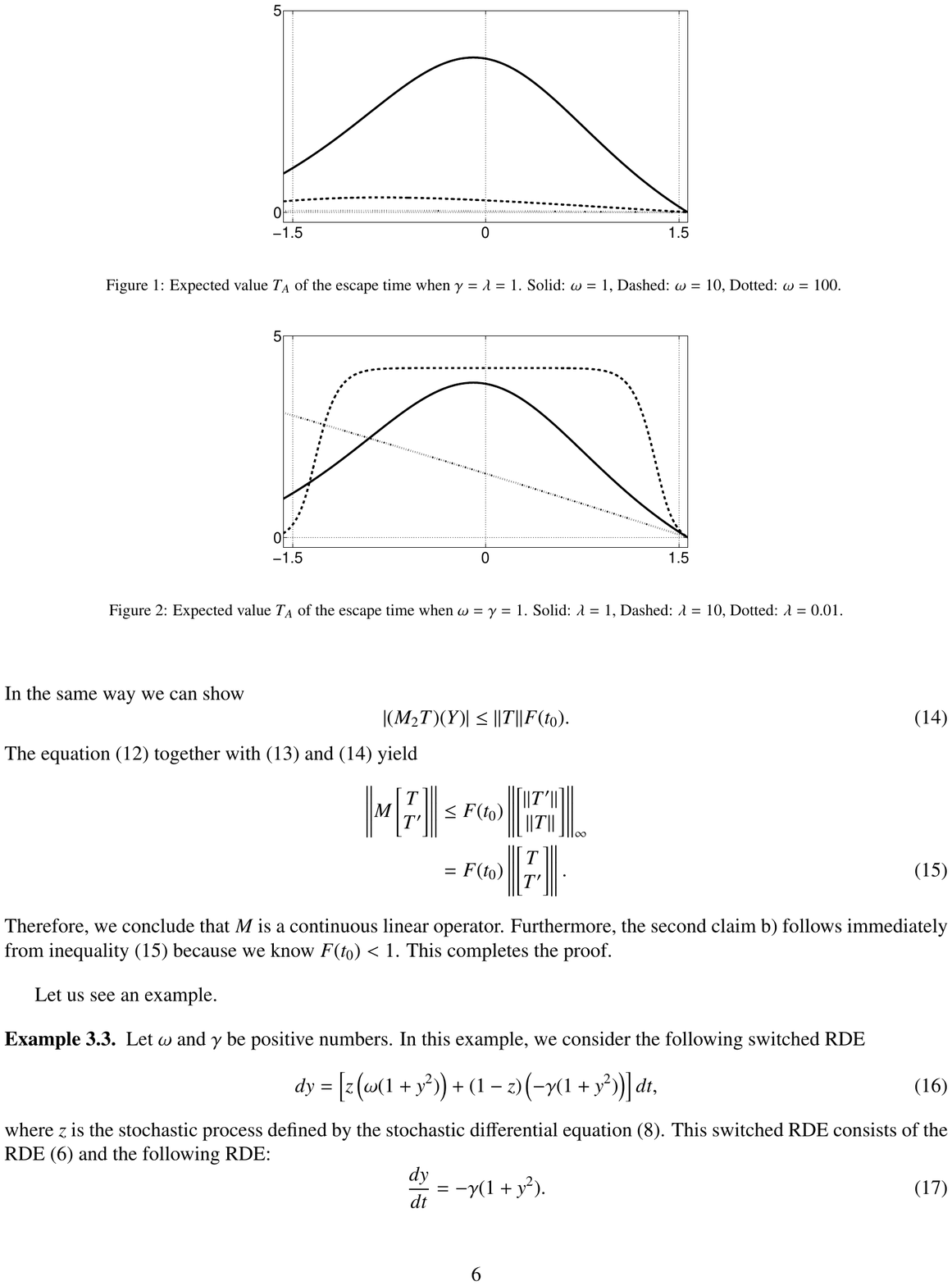}
\caption{Expected value~$T_A$ of the escape time when
$\gamma = \lambda = 1$. Solid: $\omega = 1$, Dashed: $\omega = 10$,
Dotted: $\omega = 100$. }
\label{Figure:MoveOmega}
\end{center}
% \end{figure}
% \begin{figure}[tbp]
\begin{center}
\includegraphics[width=\gwidth]{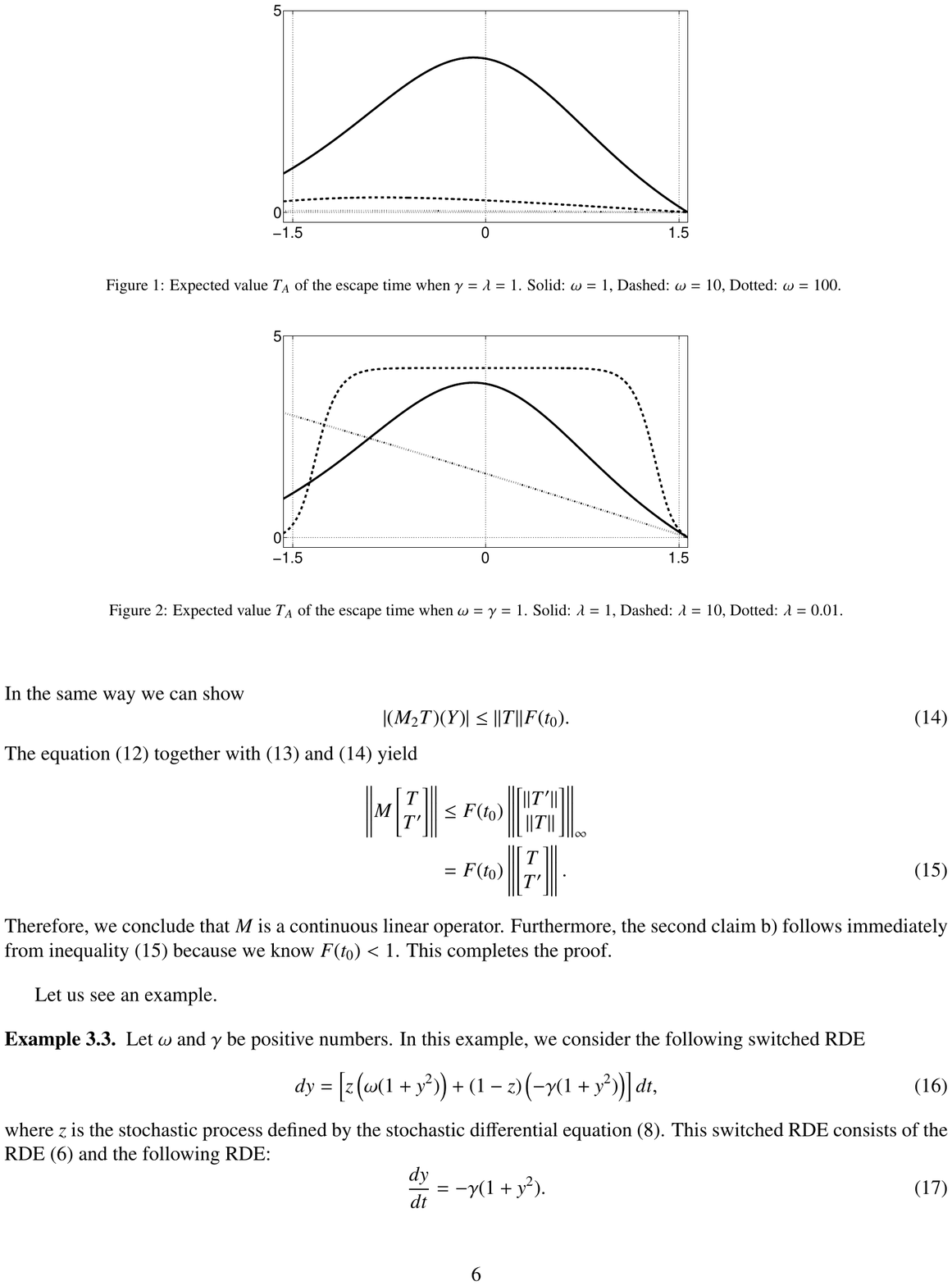}
\caption{Expected value~$T_A$ of the escape time when $\omega =\gamma = 1$.
Solid: $\lambda = 1$, Dashed: $\lambda = 10$, Dotted: $\lambda= 0.01$.}
\label{Figure:MoveLambda}
\end{center}
\end{figure}

Let us see an example. 

\begin{example}
Let $\omega$ and $\gamma$ be positive numbers. In this example, we consider the following switched RDE
\begin{equation}\label{eq:switchedRDEexample}
dy
=
\left[z\left(\omega (1+y^2)\right) + (1-z)\left(-\gamma(1+y^2)\right)\right]dt, 
\end{equation}
where $z$ is the stochastic process defined by the stochastic differential equation~\eqref{eq:z:SwitchingSignal}. This switched RDE consists of the 
RDE~\eqref{eq:ExampleRDE1} and the following RDE: 
\begin{equation} \label{eq:ExampleRDE2}
\frac{dy}{dt} = -\gamma (1+y^2).
\end{equation}
By regarding the RDEs as the local expressions of the rotations in~$G^1(\mathbb{R}^2)$ with the
angular speeds $\omega$ and $-\gamma$, we can find 
\begin{align*}%\label{eq:EscapeTime1}
t_A(\theta) = \frac{(\pi/2) - \theta}{\omega},
\quad 
t_B(\theta) = \frac{\theta + (\pi/2)}{\gamma}
\notag
\end{align*}
where an element
in~$G^1(\mathbb{R}^2)$ is identified with its angle~$\theta$ measured from the positive $x$\nobreakdash-axis. This implies that the assumption~\eqref{eq:EscTimeBounded} is satisfied
with the constant~$t_0 = \max\{\pi/\omega,\pi/\gamma\}$.

We now use Theorem~\ref{thm:analytic} to find the mean escape time of the switched RDE~\eqref{eq:switchedRDEexample}. First, we fix $\gamma =\lambda = 1$ and change $\omega$ as $\omega = 1$, $10$, and $100$. In Fig.~\ref{Figure:MoveOmega}, we show the mean escape time~$T_A$ 
as a function of the initial angle~$\theta_0 \in [-\pi/2, \pi/2]$. The mean escape time in the figure is obtained by terminating the power series~\eqref{eq:PowerSeriesSolution} at
its 21st term. It can be observed that the mean escape time decreases in $\omega$. This is  because, the larger $\omega$, the 
earlier the rotation of the line to pass the critical line of~$y$\nobreakdash-axis, at which an escape occurs. 
We then fix $\omega = \gamma = 1$ and observe how the mean escape time depends on the switching rate~$\lambda$. In 
fig.~\ref{Figure:MoveLambda}, we show the mean escape time~$T_A$ for $\lambda = 0.01$, $1$, and~$10$. When
$\lambda = 0.01$, the mean escape time is close to the escape time of the deterministic RDE~\eqref{eq:ExampleRDE1}.
This is because, when the rate~$\lambda$ is small, a switching from the first
RDE~\eqref{eq:ExampleRDE1} to the second RDE~\eqref{eq:ExampleRDE2} rarely
occurs and, hence, the presence of switching is negligible.
\end{example}

\section{Approximate Computation of Mean Escape Time}\label{subsec:approx}

One of the potential difficulties in using the analytic formula in Theorem~\ref{thm:analytic} to compute the mean escape time of switched RDE~\eqref{eq:SwitchedRDE} is in computing the escape time
of deterministic RDEs~\eqref{eq:RDE} and \eqref{eq:BRDE}. For example, consider the RDE~\eqref{eq:RDE} on $\mathbb{R}^{1\times 2}$
determined by the matrix
\begin{equation*}
A = T^{-1}
\begin{bmatrix}
0   &  1   &  0  \\
-1  &  0   &  0  \\
0   &  0   &  1
\end{bmatrix}T,
\quad
T = \begin{bmatrix}
1  &  0  &  1  \\
0  &  1  &  0  \\
0  &  0  &  1
\end{bmatrix}
\end{equation*}
with the initial state
\begin{equation*}
Y_0 = \begin{bmatrix} 1 \\ 1 \end{bmatrix}.
\end{equation*}
We can show that the corresponding ERDE is the flow
\begin{equation*}
e^{At}\left(\psi(Y_0)\right)
=
\left\{
r
\begin{bmatrix}
2\cos t + \sin t - e^{t}  \\
\cos t - 2\sin t  \\
e^t
\end{bmatrix}
:
r\in \mathbb{R}
\right\}
\end{equation*}
in $P(\mathbb{R}^3)$. Therefore, under the canonical chart~$\psi$, this RDE
escapes when $2\cos t + \sin t - e^{t} = 0$. However it is not easy to calculate the zeros of this type of transcendental equation effectively.

The objective of this section is to provide an approximative method to compute the mean escape time of the switched RDE~\eqref{eq:SwitchedRDE}. In Subsection~\ref{sec:FixedInitialState}, we propose a method to approximately compute the escape time of the deterministic RDEs. Although this method allows us to use Theorem~\ref{thm:analytic} to approximately compute the mean escape time of the switched RDE~\eqref{eq:SwitchedRDE}, it is not clear if the computation is robust with respect to the computational error of the escape time of the deterministic RDEs. Therefore, in Subsection~\ref{sec:net}, we present a theorem confirming the computational robustness.

\subsection{Sequence converging to escape time}
\label{sec:FixedInitialState}

In this section, we give a procedure for approximately computing the escape time
of deterministic RDEs. We remark that, although there are many results concerning the occurrence of finite-time
escape phenomenon, the computation of escape time has not attracted much
attention. The result by~\cite{Getz1977} can be applied to only
symmetric RDEs. The lower estimate of the
escape time by~\cite{Jodar1995} can be applied for any Riccati
differential equation but their estimate tend to be conservative. 

The aim of this section is to prove the following theorem, which enables us to approximately
calculate the escape time of the deterministic RDE~\eqref{eq:RDE} with an arbitrary precision. Before stating the theorem, let us recall that the principal branch of the Lambert $W$-function (see, e.g., \cite{CorlessGonnetHareJeffreyKnuth1996}) $W(\cdot)$ is defined as the
inverse of the mapping $[0, \infty) \to [0, \infty) \colon x \mapsto
xe^x$.

\begin{theorem}\label{thm:ConvSequence}
Let $Y_0 \in \mathbb{R}^{(d-k)\times k}$ be arbitrary. Assume that $t_A(Y_0) < \infty$. Define the function $\Delta \colon [0, t_A(Y_0)) \to (0,
\infty)$ by 
\begin{equation*}%\label{eq:def:delta}
\Delta(t) = W\left(\frac{\norm A}{\Norm{\begin{bmatrix}
I&0
\end{bmatrix}A}\Norm{\begin{bmatrix}
I\\Y(t;Y_0)
\end{bmatrix}}} \right)\frac{1}{\norm A}. 
\end{equation*}
Then, the real sequence $\{t_n\}_{n=0}^\infty$ defined by 
$t_0 = 0$ and the difference equation 
\begin{equation*}
t_{n+1} = t_n + \Delta(t_n),\ n=0, 1, \dotsc
\end{equation*}
satisfies
\begin{equation*}
\lim_{n\to\infty}t_n = t_A(Y_0). 
\end{equation*}
\end{theorem}

In the rest of this subsection, we present the proof of Theorem~\ref{thm:ConvSequence}. 
We start by recalling the following lemma, which was implicitly stated in~\cite{Jodar1995} and shall be proved in this paper for the sake of completeness. 

\begin{lemma}\label{lem:Jodar}
Let $\Delta$ be as in Theorem~\ref{thm:ConvSequence}. If $t\in \left[0,
t_A(Y_0)\right)$ then $[t, t+\Delta(t)]\subset \left[0,
t_A(Y_0)\right)$. 
\end{lemma}

\begin{pf}
Without loss of generality we can assume $t=0$ because otherwise we can
reduce the problem to the case $t=0$ by considering the Riccati
differential equation \eqref{eq:RDE} with the initial state $Y(t;Y_0)$. 

Let $t=0$. We need to show that the matrix $U(s)$ defined by
\eqref{eq:def:UV} is invertible if $0\leq s\leq \Delta(0)$. Since $U(0)
= I$ is clearly invertible we can assume $0<s\leq \Delta(0)$. By
Lemma~\ref{lemma:SmallGain} it is sufficient to show that $\norm{U(s)-I}
< 1$. Since 
\begin{equation*}
U(s)
=
\begin{bmatrix}
I_k&0
\end{bmatrix}e^{As}\begin{bmatrix}
I_k\\Y_0
\end{bmatrix}, 
\end{equation*}
an easy computation shows that 
\begin{equation*}
U(s) - I_k
=
\begin{bmatrix}
I_k&0
\end{bmatrix}(e^{As}-I_{d})\begin{bmatrix}
I_k\\Y_0
\end{bmatrix}. 
\end{equation*}
Therefore, by Lemma~\ref{lemma:ExpEstimate}, if $0<s\leq \Delta(0)$ then 
\begin{equation*}
\begin{aligned}
\norm{U(s) - I_k}
&<
\Norm{
\begin{bmatrix}
I_k&0
\end{bmatrix}A
}
\abs{s}e^{\norm{A}\abs{s}}\Norm{
\begin{bmatrix}
I_n\\Y_0
\end{bmatrix}
}
\\
&\leq
\Norm{
\begin{bmatrix}
I_k&0
\end{bmatrix}A
}
\Delta(0)^{\norm{A}\Delta(0)}\Norm{
\begin{bmatrix}
I_k\\Y_0
\end{bmatrix}
}
\\
&= 1, 
\end{aligned}
\end{equation*}
where in the last equation we used the definition of~$\Delta$ and the
Lambert $W$ function. 
\end{pf}

We will also use the following estimate of matrix exponentials. 

\begin{lemma}\label{lemma:ExpEstimate}
Let $M$ and $A$ be nonzero square matrices with the same dimension. If
$t\neq 0$  then
% \begin{equation*}
$\norm{M(e^{At}-I)} < \norm{MA} \abs{t} e^{\norm{A}\abs{t}}$. 
% \end{equation*}
\end{lemma}

\begin{pf}
By the definition of exponential matrices we have 
$M(e^{At}-I) = MAt \sum_{i=1}^\infty \frac{A^{i-1}{t}^{i-1}}{i!}$. 
Therefore it easily follows that 
\begin{equation*}
\begin{aligned}
\norm{M(e^{At}-I)}
&\leq
\norm{MA}\abs{t} \sum_{i=0}^\infty \frac{\norm{A}^{i}\abs{t}^{i}}{(i+1)!}
\\
&<
\norm{MA}\abs{t} \sum_{i=0}^\infty \frac{\norm{A}^{i}\abs{t}^{i}}{i!}
\\
&=
\norm{MA}\abs{t} e^{\norm{A}\abs{t}}. 
\end{aligned}
\end{equation*}
\end{pf}

Now we are ready to prove Theorem~\ref{thm:ConvSequence}. 

\begin{pf}[Theorem~\ref{thm:ConvSequence}]
Notice that the function $\Delta$ is continuous because both $W$ and $Y$ are
continuous. Moreover, since $W(x) = 0$ if and only if $x=0$, the
function $\Delta$ never vanishes. 
Now, by Lemma~\ref{lem:Jodar}, the real sequence $\{t_n\}_{n=0}^\infty$ is
bounded from above by $t_A(Y_0)<\infty$. Also $\{t_n\}_{n=0}^\infty$ is
increasing by its definition. Therefore the sequence has a limit, say,
$\tilde t \in \mathbb{R}$. Lemma~\ref{lem:Jodar} immediately shows
$\tilde t \leq t_A(Y_0)$. Assume $\tilde t < t_A(Y_0)$ to derive a
contradiction. Since the sequence $\{t_n\}_{n=0}^\infty$ is convergent,
the sequence~$\{\Delta(t_n)\}_{n=0}^\infty$ must  converge to~$0$. By
the continuity of~$\Delta$ on $[0, t_A(Y_0))$ we have $\Delta(\tilde t)
= 0$, but this contradicts to the fact  that $\Delta$ never vanishes on
$[0, t_A(Y_0))$. Hence $\tilde t = t_A(Y_0)$. 
\end{pf}

Let us see examples. 

\begin{example} \label{exmp:scalarfixed}
Consider the scalar RDE 
\begin{equation}\label{eq:RDE:dWdt=2W+W^2}
\frac{dy}{dt} = y^2 + 2y
\end{equation}
with the initial state $y(0) = 1$. This RDE
is associated with the matrix
\begin{equation*}%\label{eq:scalar:A_matrix}
A = \begin{bmatrix}
-1&-1\\0&1
\end{bmatrix}. 
\end{equation*}
Since $U(t) = (3e^{-t} - e^t)/2$, we can check that the differential
equation  escapes precisely at $t = (\log 3)/2 = 0.549306$. The sequence
$\{t_n\}_{n=1}^\infty$ in Theorem~\ref{thm:ConvSequence} gives $t_{10} =
0.549290$ and $t_{20} = 0.549306$. 
Fig.~\ref{fig:EscapeTimeScalarFixedInit} shows the graphs of~$t_n$ and
$\log(\Delta(t_n))$. 

\begin{figure}[tb] 
\centering
\includegraphics[width=\mywidth]{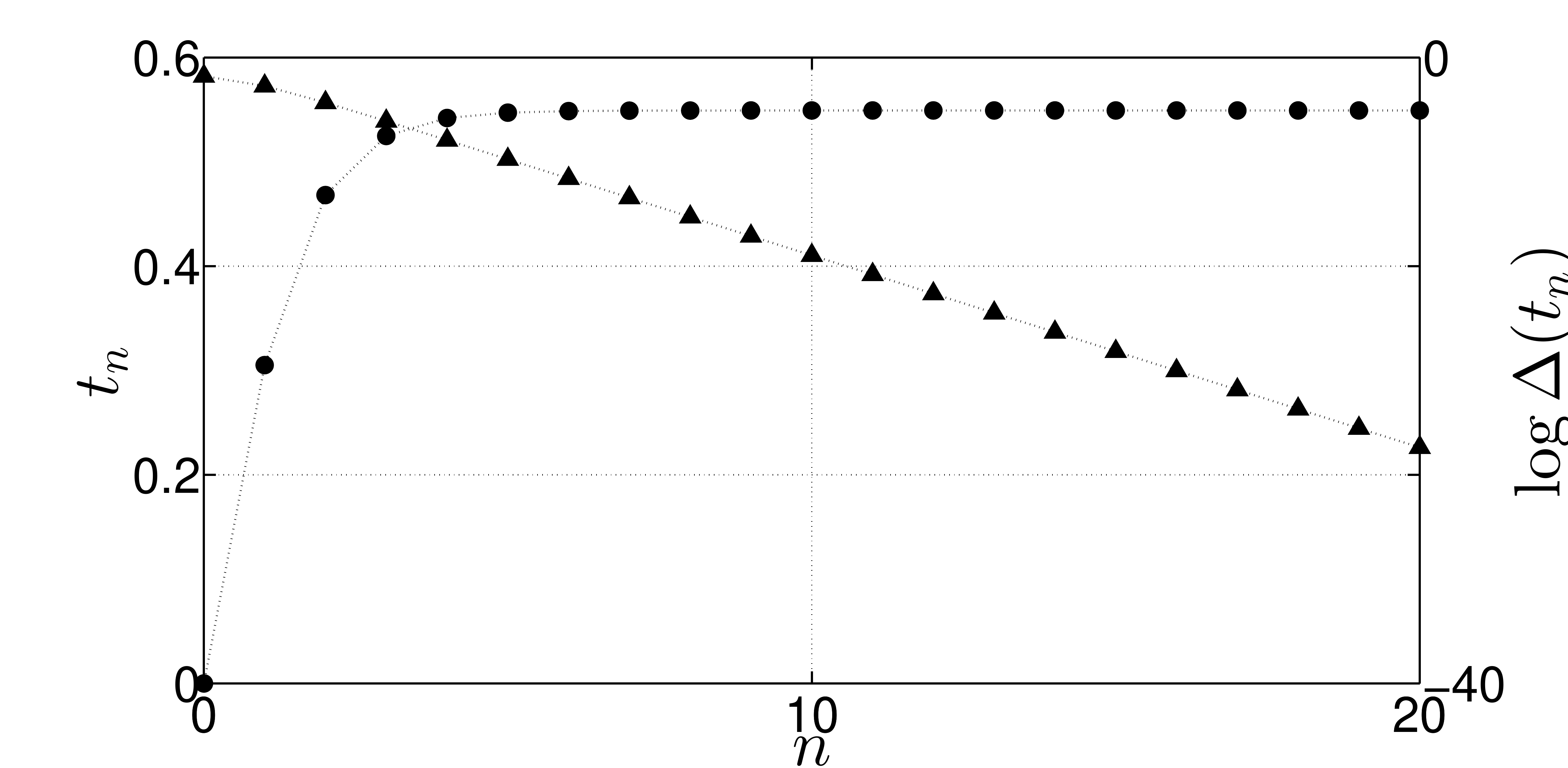}
\caption{Escape time of RDE \eqref{eq:RDE:dWdt=2W+W^2}. Circle: $t_n$. Triangle: $\log\Delta(t_n)$}
\label{fig:EscapeTimeScalarFixedInit}
\end{figure}
\end{example}

\begin{example}
Consider the vector-valued RDE
\begin{equation}\label{eq:RDE:exmplevec}
  \frac{dY}{dt} =
  \begin{bmatrix}
0\\-1
  \end{bmatrix}
  +
  \begin{bmatrix}
  2&-1\\3&-3
  \end{bmatrix}Y + Y - Y\begin{bmatrix}
2&-1
  \end{bmatrix}Y
\end{equation}
with the initial condition 
\begin{equation*}
Y_0 = \begin{bmatrix}-0.5\\-0.5\end{bmatrix}. 
\end{equation*}
Fig.~\ref{fig:EscapeTimeVecFixedInit} shows the graph of~$t_n$ and
$\log(\Delta(t_n))$. Notice that in this case the scalar function $U(t)$
has a complicated form and it is not as easy to find its zeros as in
Example~\ref{exmp:scalarfixed}. 

\begin{figure}[tb] 
\centering
\includegraphics[width=\mywidth]{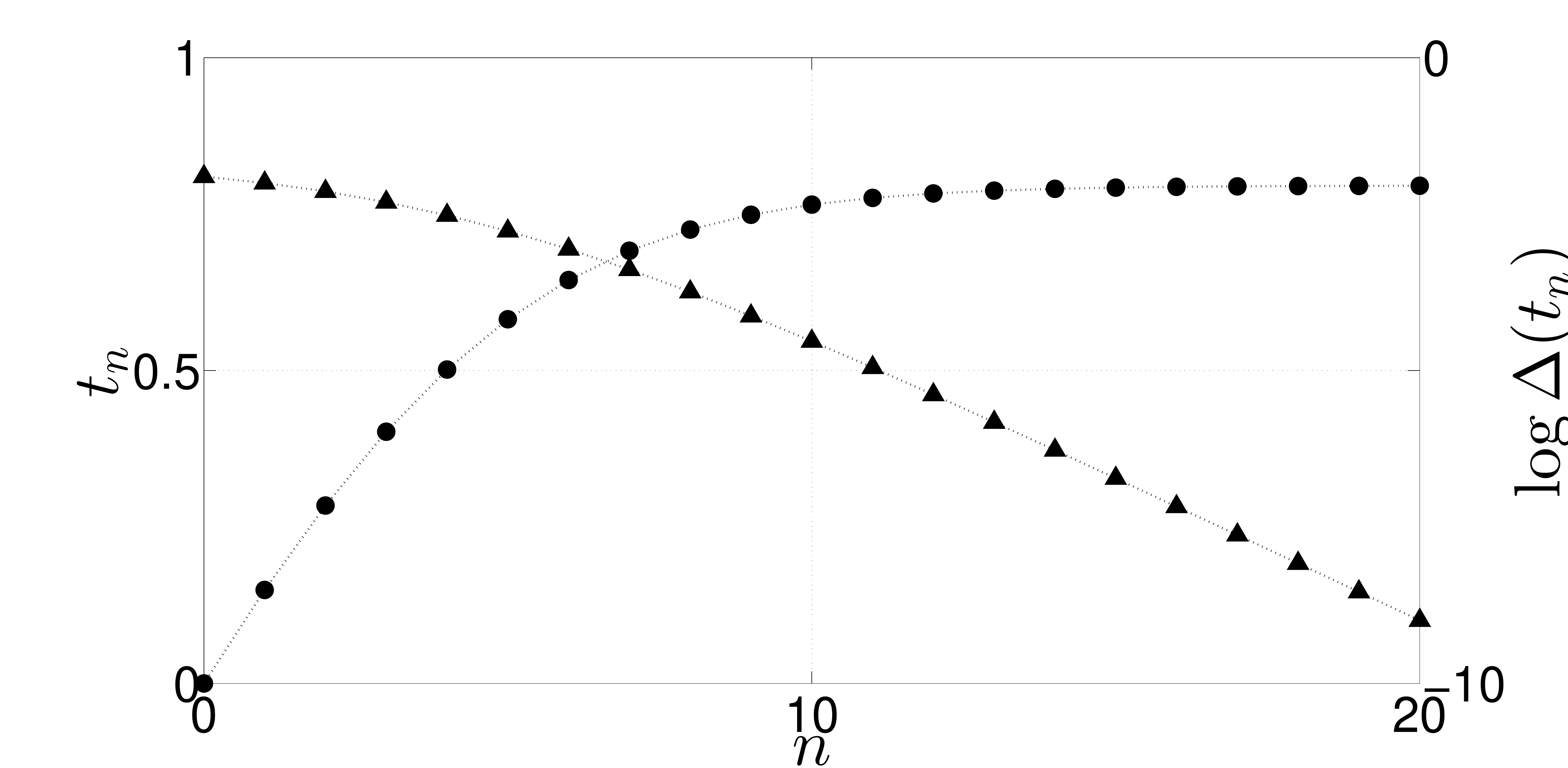}
\caption{Escape time of RDE \eqref{eq:RDE:exmplevec}. Circle: $t_n$. Triangle: $\log\Delta(t_n)$}
\label{fig:EscapeTimeVecFixedInit}
\end{figure}
\end{example}

Theorem~\ref{thm:ConvSequence} allows us to approximately compute the escape time of the deterministic RDEs for a fixed initial condition. On the other hand, to use Theorem~\ref{thm:analytic} for the computation of the mean escape time of the switched RDE~\eqref{eq:SwitchedRDE}, we need the escape time of each of the deterministic RDEs~\eqref{eq:RDE} and \eqref{eq:BRDE} with arbitrary initial conditions. However, it is not feasible to apply the theorem to all the possible initial conditions. To fill in this gap, the following trivial lemma is useful. 

\begin{lemma}\label{lem:}
 For any $Y_0 \in \mathbb{R}^{(d-k)\times k}$ and  $0\leq t<t_A(Y_0)$ we
have $t_A(Y(t;Y_0)) = t_A(Y_0) - t$. 
\end{lemma}

Lemma~\ref{lem:} suggests that one instance of the computation of the escape time for a single initial condition allows us to find the escape time with various initial conditions. Specifically,
finding the escape time for an initial state~$Y_0$ by using Theorem~\ref{thm:ConvSequence} gives us the
escape time for a \emph{family} of initial states
$\{Y(t;Y_0)\}_{t= t_0, t_1, \dotsc}$. This observation leads us to
Algorithm~\ref{alg:function}, which calculates the escape time for
several many initial states efficiently. 

\begin{algorithm}[tb]
\caption{Calculation of the escape time as a function of initial states}
\label{alg:function}
\begin{algorithmic}[1]
\STATE Let $N$ be a positive integer. $D = \emptyset$.
\REPEAT 
\STATE Randomly choose an initial state $Y_0$. 
\STATE Compute the sequence $\{t_n\}_{n=0}^N$
\STATE $D \leftarrow D \cup \{( Y(t_N-t_n, Y_0), t_n\}_{n=0}^N$
\UNTIL Projection of $D$ into its first component becomes dense in the set of the initial states. 
\end{algorithmic}
\end{algorithm}

We illustrate the effectiveness of Algorithm~\ref{alg:function} in the following example. 

\begin{figure}[tb] 
\centering
\includegraphics[width=\mywidth]{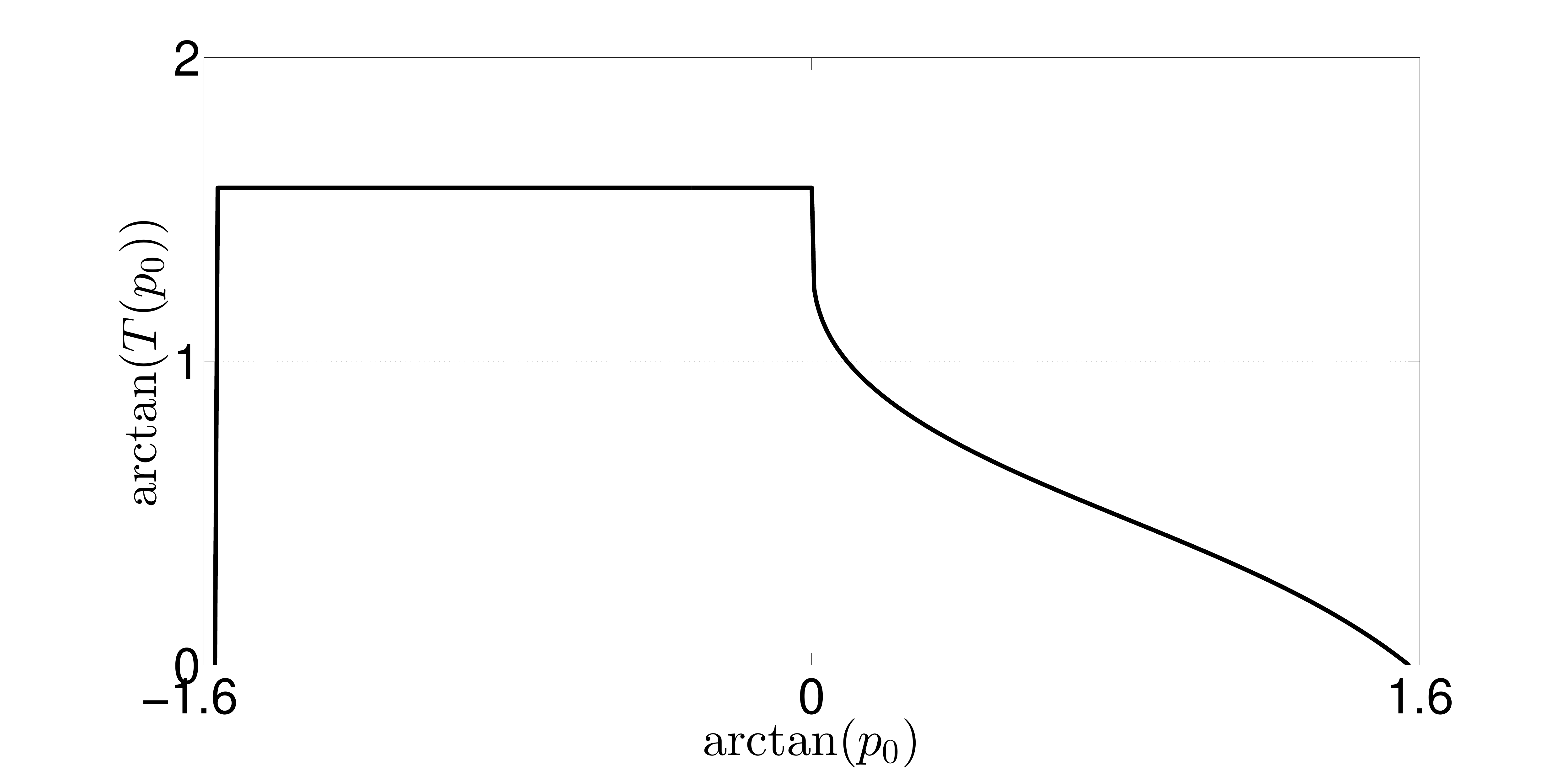}
\caption{Escape time of the RDE \eqref{eq:RDE:dWdt=2W+W^2}}
\label{fig:escapetimefunctiondim1}
\vspace{3mm}
\includegraphics[width=\myvecwidth]{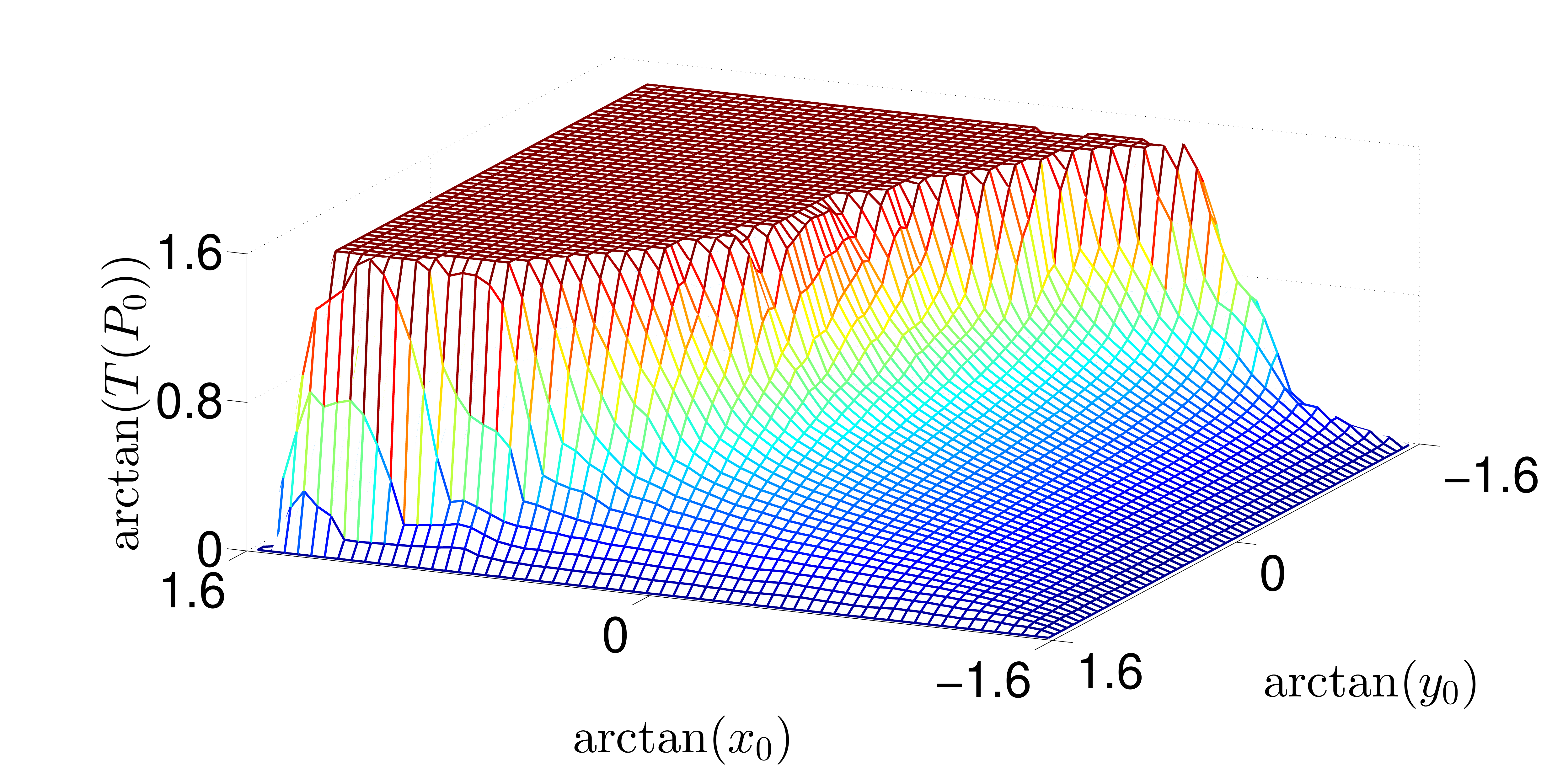}
\caption{Escape time of the RDE \eqref{eq:RDE:exmplevec}}
\label{fig:EscapeTime2dim}
\end{figure}

\begin{example}
We first apply Algorithm~\ref{alg:function} to the scalar RDE~\eqref{eq:RDE:dWdt=2W+W^2}. The algorithm gives the graph of the escape
time as a function of initial states as in 
Fig.~\ref{fig:escapetimefunctiondim1}. Notice that, the real axis
$\mathbb{R}$ (or the half real axis $[0, \infty)$) are rescaled to the
closed interval $[-\pi/2, \pi/2]$ ($[0, \pi/2]$, respectively) with the
arc-tangent function for the ease of presentation. The flatten part in
the left half, where the graph takes the value $\pi/2$, means that
the RDE never escapes in a finite time. 

We then consider again the vector-valued Riccati differential
equation~\eqref{eq:RDE:exmplevec}. The associated matrix 
\begin{equation*}
A = \begin{bmatrix}
-1 & 2 & -1\\
0 & 2 & -1\\
-1 & 3 & -3
\end{bmatrix}
\end{equation*}
has the real and distinct eigenvalues $1.4605$, $-0.7609$, and
$-2.6996$. Using Algorithm~\ref{alg:function} we can draw the graph of
the escape time as a function of initial states as in
Fig.~\ref{fig:EscapeTime2dim}. Notice that the values are rescaled by the arc-tangent function as was did in 
Fig.~\ref{fig:escapetimefunctiondim1}.  \end{example}

\subsection{A robustness result}\label{sec:net}

Although Theorems~\ref{thm:analytic} and~\ref{thm:ConvSequence} could allow us to approximately and numerically compute the mean escape time of the switched RDE~\eqref{eq:SwitchedRDE}, it is not clear if the computation is robust with respect to the computational error in Theorem~\ref{thm:ConvSequence} for the escape time of the deterministic RDEs. To provide an affirmative answer to this question, in this subsection we present a theorem showing that taking sufficiently many escape time using Algorithm~\ref{alg:function} can lead to provide an accurate estimate of the mean escape time of the switched RDE~\eqref{eq:SwitchedRDE}. 

Throughout this section, we will identify all the matrices $Y$ in
$\mathbb{R}^{(d-k)\times k}$ with its image $\psi(Y) \in
G_0^{k}(\mathbb{R}^{d})$ by the canonical chart. Since a Grassmann manifold
is compact, the manifold is totally bounded. Therefore, the subset
$G_0^{k}(\mathbb{R}^{d})$ is also totally bounded and hence admits a
finite $\epsilon$-net for every $\epsilon > 0$. This fact will allow us 
to discretize the state space with finitely many points. 

In this subsection, in addition to Assumption~\ref{asm:boudnedneddsss}, we also place the following technical assumption. 

\begin{assumption}\label{assum:UnifCont}
The mean escape times $T_A$ and $T_B$ are uniformly continuous as
functions from $G_0^k(\mathbb{R}^{d})$ to $\mathbb{R}$. 
\end{assumption}

Under this assumption, the following theorem states that an accurate estimate of the mean escape time of the switched RDE~\eqref{eq:SwitchedRDE} is possible by finitely many values of the escape time of deterministic RDEs~\eqref{eq:RDE} and~\eqref{eq:BRDE}. 

\begin{theorem}\label{thm:main:numerical}
Let $\epsilon>0$ be arbitrary. Then, there exists a finite set $S =
\{s_1, \dotsc, s_L\} \subset G^k_0(\mathbb{R}^{d})$ and an invertible matrix $\Psi$ such that 
\begin{equation*}
\Norm{
\begin{bmatrix}
T_A(S)\\T_B(S)
\end{bmatrix}
-
\Psi^{-1} \begin{bmatrix}
g_A(S)\\g_B(S)
\end{bmatrix}
}_\infty
<
\epsilon. 
\end{equation*}
\end{theorem}

To prove Theorem~\ref{thm:main:numerical}, we need the following lemma, which enables us to approximate the operators $M_A$ and $M_B$
by matrices. 

\begin{lemma}\label{lem:epsilon}
Let $\epsilon> 0$ be arbitrary. Then, there exists a finite set $S = \{s_1, \dotsc,
s_L\} \subset G^k_0(\mathbb{R}^{d})$ and matrices $N_A, N_B$ such that
\begin{equation*}
\begin{aligned}
\Norm{
(M_AT_B)(S) - N_A(T_B(S))}_\infty
&< \epsilon, \\
\Norm{
(M_BT_A)(S) - N_B(T_A(S))}_\infty
&< \epsilon
\end{aligned}
\end{equation*}
where, for a function $f \colon G^k_0(\mathbb{R}^{d}) \to \mathbb{R}$, 
we write 
\begin{equation*}
f(S) = \begin{bmatrix}
f(s_1)\\\vdots\\f(s_L)
\end{bmatrix} \in \mathbb{R}^L. 
\end{equation*}
Moreover, the matrices satisfy $\norm{N_A}_\infty, \norm{N_B}_\infty \leq F(t_0)$. 
\end{lemma}

\begin{pf}
Let us first recall that the Grassmannian~$G^k(\mathbb{R}^{d})$ can be equipped with the following distance; If $V_1, V_2
\in G^n(\mathbb{R}^{d})$ and $\pi_1, \pi_2$ are orthogonal projections
from $\mathbb{R}^{d}$ onto $V_1, V_2$, respectively, then one can
define a metric on $G^n(\mathbb{R}^{d})$ as $\rho(V_1, V_2) =
\norm{\pi_1-\pi_2}$. 

Now, by symmetry, it is sufficient to prove only the first inequality in the lemma. Let
$\epsilon > 0$ be arbitrary. By Assumption~\ref{assum:UnifCont} there
exists $\delta>0$ such that if $\rho(P, P') < \delta$, then $\abs{T_A(P)
- T_A(P')} < \epsilon/2$. Then, we take a finite $\delta$-net $X = \{s_1,
\dotsc, s_L\}$ of~$G^k_0(\mathbb{R}^{d})$. Let $q$ be a mapping on
$G^k_0(\mathbb{R}^{d})$ that assigns to any $P$ one of~$s_\ell$ within
the distance of~$\delta$. 

Then $h>0$ and define the mapping $M_A'T_B\colon
G^k_0(\mathbb{R}^{d}) \to \mathbb{R}$ by
\begin{equation*}
\begin{aligned}
(M'_A T_B)(P) &= \sum_{n=0}^\infty T_B(P_A(nh; P))\xi_{A, n}(P), \\
\xi_{A, n}(P) &= \int_{nh}^{(n+1)h}  f(\tau) \chi_{[0, t_A(P))}(\tau)\,d\tau. 
\end{aligned}
\end{equation*}
Since $L$ is finite we can take a sufficiently small $h>0$ such that,
for every $\ell=1, \dotsc, L$,
\begin{equation}\label{eq:MATB-M'ATB}
\Abs{(M_AT_B)(s_\ell) - (M'_A T_B)(s_\ell)} < \epsilon/2. 
\end{equation}
Let us fix such $h$. Then let us define $M_A''T_B\colon
G^k_0(\mathbb{R}^{d}) \to \mathbb{R}$ by
% \begin{equation*}
$(M''_A T_B)(P) = \sum_{n=0}^\infty T_B(q(P_A(nh; P)))\xi_{A, n}(P)$. 
% \end{equation*}
Since $\rho(P, q(P)) < \delta$ for every $P \in
G^k_0(\mathbb{R}^{d})$, by the uniform continuity of~$T_B$,  for all
$\ell=1, \dotsc, L$ it holds that 
% \begin{equation*}
$\abs{(M'_AT_B)(s_\ell) - (M''_A T_B)(s_\ell)} < \epsilon/2$. 
% \end{equation*} 
This inequality together with \eqref{eq:MATB-M'ATB} implies
\begin{equation*}
\norm{(M_AT_B)(S) - (M''_A T_B)(S)}_\infty < \epsilon. 
\end{equation*}
Now, since $(M''_A T_B)(W_k)$ is a linear combination of~$T_B(s_1)$,
$\dotsc$, $T_B(s_L)$  (notice that the image of~$q$ is always one of
$s_\ell$, $1\leq \ell \leq L$) there exists a matrix $N_A$ such that $(M''_A T_B)(S)
= N_A(T_B(S))$. This completes the proof of the first claim. 

To prove the second claim, we just need to notice that the sum of the
$\ell$th row of the matrix $N_A$ is equal to 
% \begin{equation*}
$\sum_{n=0}^\infty  \xi_{A, n}(s_\ell) = \int_0^{t_A(s_\ell)}f(\tau) \,d\tau$, 
% \end{equation*}
which is less always than or equal to $F(t_0)$ by the assumption
\eqref{eq:EscTimeBounded}.
\end{pf}

We can now prove  Theorem~\ref{thm:main:numerical}. 

\begin{pf}[Theorem~\ref{thm:main:numerical}]
First let us apply Lemma~\ref{lem:epsilon} with the constant $F(t_0)
\epsilon$ to obtain a finite set $S = \{s_1, \dotsc, s_L\}$ and matrices
$N_A, N_B$. Define
 \begin{equation*}
\Psi = \begin{bmatrix}
I& -N_A\\-N_B& I
\end{bmatrix}. 
\end{equation*}
By Lemmas~\ref{lemma:SmallGain} and~\ref{lem:epsilon}, we can see that this
matrix $\Psi$ is invertible and satisfy $\norm{\Psi^{-1}} \leq
1/F(t_0)$. Since 
\begin{equation*}
\begin{multlined}
\begin{bmatrix}
T_A(S)\\T_B(S)
\end{bmatrix}
-
\Psi^{-1} \begin{bmatrix}
g_A(S)\\g_B(S)
\end{bmatrix}
% \\
=
\Psi^{-1} \begin{bmatrix}
(M_AT_B)(S)-N_A(T_B(S))\\
(M_BT_A)(S)-N_B(T_A(S))
\end{bmatrix}
\end{multlined}
\end{equation*}
we have
\begin{equation*}
\begin{aligned}
\Norm{
\begin{bmatrix}
T_A(S)\\T_B(S)
\end{bmatrix}
-
\Psi^{-1} \begin{bmatrix}
g_A(S)\\g_B(S)
\end{bmatrix}}
<
\norm{\Psi^{-1}} (F(t_0)\epsilon)
% \\
\leq
\epsilon, 
\end{aligned}
\end{equation*}
as desired. This completes the proof of the theorem. 
\end{pf}

\section{Conclusion}\label{sec:conc}

In this paper, we have investigated the computational aspects of the mean escape time of the switched RDEs subject to Poisson switching. We have first shown that, if the escape times of their subsystems are available, then we can find the mean escape time of the switched RDE as a convergent power series. We have then presented an approximation framework for numerically computing the escape time of RDEs as the limit of a convergent sequence, which can enhance the applicability of the characterization as a power series. Numerical simulations are presented to illustrate the obtained results. 

\section*{Acknowledgment}

This work was supported by JSPS KAKENHI Grant Number JP21H01352.


\begin{thebibliography}{10}
\expandafter\ifx\csname url\endcsname\relax
  \def\url#1{\texttt{#1}}\fi
\expandafter\ifx\csname urlprefix\endcsname\relax\def\urlprefix{URL }\fi
\expandafter\ifx\csname href\endcsname\relax
  \def\href#1#2{#2} \def\path#1{#1}\fi
  
\bibitem{Shayman1986}
M.~A. Shayman, {Phase portrait of the matrix Riccati equation}, SIAM Journal on
  Control and Optimization 24~(1) (1986) 1--65.

\bibitem{Basar1989}
T.~Ba\c{s}ar, {Generalized Riccati equations in dynamics
  games}, in: The Riccati Equation, Springer-Verlag, New York, 1989, pp.
  293--333.

\bibitem{Doyle1989}
J.~C. Doyle, K.~Glover, P.~P. Khargonekar, B.~A. Francis, {State space
  solutions to standard $H_2$ and $H_\infty$
  control problems}, IEEE Transactions on Automatic Control 34 (1989) 831--847.

\bibitem{Chang1972}
K.~W. Chang, \href{http://epubs.siam.org/doi/pdf/10.1137/0503050}{{Singular
  perturbations of a general boundary value problem}}, SIAM Journal on
  Mathematical Analysis 3~(3) (1972) 520--527.
% \newline\urlprefix\url{http://epubs.siam.org/doi/pdf/10.1137/0503050}

\bibitem{Freiling2002}
G.~Freiling, {A survey of nonsymmetric Riccati equations}, Linear Algebra and
  its Applications 351-352 (2002) 243--270.
  
\bibitem{Watson1995}
G.~N. Watson, {A Treatise on the Theory of Bessel Functions}, Cambridge
  University Press, 1995.

\bibitem{Martin1981}
C.~Martin, {Finite escape time for Riccati differential equations}, Systems
  \& Control Letters 1~(2) (1981) 127–131.

\bibitem{Sasagawa1982}
T.~Sasagawa, {On the finite escape phenomena for matrix Riccati equations},
  IEEE Transactions on Automatic Control 27~(4) (1982) 977--979.

\bibitem{Crouch1987}
P.~E. Crouch, M.~Pavon, {On the existence of solutions of the Riccati
  differential equation}, Systems \& Control Letters 9 (1987)
  203--206.
  
\bibitem{Doolin1990}
B.~F. Doolin, C.~F. Martin, {Introduction To Differential Geometry For
  Engineers}, Marcel Dekker Inc, 1990.


\bibitem{Getz1977}
W.~M. Getz, D.~H. Jacobson, {Sufficiency conditions for finite escape times in
  systems of quadratic differential equations}, J. Inst. Maths Applics 19
  (1977) 377--383.

\bibitem{Freiling2000}
G.~Freiling, G.~Jank, A.~Sarychev, {Non-blow-up conditions for Riccati-type
  matrix differential and difference equations}, Results in Mathematics 37
  (2000) 84--103.

\bibitem{Zhang2008}
L.~Zhang, E.-K. Boukas, J.~Lam,
  \href{http://ieeexplore.ieee.org/lpdocs/epic03/wrapper.htm?arnumber=4668537}{{Analysis
  and synthesis of Markov jump linear systems with time-varying delays and
  partially known transition probabilities}}, IEEE Transactions on Automatic
  Control 53~(10) (2008) 2458--2464.


\bibitem{Feng1992}
X.~Feng, K.~Loparo, Y.~Ji, H.~Chizeck, {Stochastic stability properties of jump
  linear systems}, IEEE Transactions on Automatic Control 37 (1992) 38--53.

\bibitem{Zhang2010}
L.~Zhang, J.~Lam, {Necessary and sufficient conditions for analysis and
  synthesis of Markov jump linear systems with incomplete transition
  descriptions}, IEEE Transactions on Automatic Control 55~(7) (2010)
  1695--1701.

\bibitem{Shi2015}
P.~Shi, F.~Li, {A survey on Markovian jump systems: Modeling and design},
  International Journal of Control, Automation and Systems 13~(1) (2015) 1--16.

\bibitem{Wu2017a}
Z.~G. Wu, S.~Dong, P.~Shi, H.~Su, T.~Huang, R.~Lu, {Fuzzy-model-based
  nonfragile guaranteed cost control of nonlinear Markov jump systems}, IEEE
  Transactions on Systems, Man, and Cybernetics: Systems 47~(8) (2017)
  2388--2397.

\bibitem{Jin2022}
L.~Jin, Y.~Yin, R.~Loxton, Q.~Lin, F.~Liu, K.~L. Teo, {Optimal control of
  nonlinear Markov jump systems by control parametrisation technique}, IET
  Control Theory and Applications (2022).

\bibitem{Zhang2019b}
M.~Zhang, P.~Shi, L.~Ma, J.~Cai, H.~Su, {Quantized feedback control of fuzzy
  Markov jump systems}, IEEE Transactions on Cybernetics 49~(9) (2019)
  3375--3384.

\bibitem{Shen2016b}
M.~Shen, J.~H. Park, D.~Ye, {A Separated Approach to Control of Markov Jump
  Nonlinear Systems with General Transition Probabilities}, IEEE Transactions
  on Cybernetics 46~(9) (2016) 2010--2018.

\bibitem{Hanlon2011a}
B.~Hanlon, V.~Tyuryaev, C.~Martin, {Stability of switched linear systems with
  Poisson switching}, Communications in Information and Systems 11~(4) (2011)
  307--326.

\bibitem{Jodar1995}
L.~Jodar, E.~Ponsoda, {Non-autonomous Riccati-type matrix differential
  equations: existence interval, construction of continuous numerical solutions
  and error bounds}, IMA journal of numerical analysis 15~(1) (1995) 61--74.

\bibitem{CorlessGonnetHareJeffreyKnuth1996}
R.~M. Corless, G.~H. Gonnet, D.~E.~G. Hare, D.~J. Jeffrey, D.~E. Knuth, {On the
  Lambert W function}, Advances in Computational Mathematics 5 (1996) 329--359.


\end{thebibliography}
\end{document}